\documentclass[a4paper,11pt]{amsart}
\usepackage{lmodern}
\usepackage{amsmath,amsfonts,amssymb,mathrsfs,amsthm,  cleveref, enumitem}

\newtheorem{theorem}[subsection]{Theorem}
\newtheorem{lemma}[subsection]{Lemma}
\newtheorem{cor}[subsection]{Corollary}
\newtheorem{definition}[subsection]{Definition}
\newtheorem{prop}[subsection]{Proposition}

\newcommand{\bC}{{\mathbb{C}}}
\newcommand{\bN}{{\mathbb{N}}}

\begin{document}


\title{Interpolation polynomials and linear algebra}

\author[Khovanskii]{Askold Khovanskii}
\author[Singla]{Sushil Singla}
\author[Tronsgard]{Aaron Tronsgard}

\maketitle


\begin{abstract}
We reconsider the theory of Lagrange interpolation polynomials with multiple interpolation points and apply it to linear algebra. For instance, $A$ be a linear operator satisfying a degree $n$ polynomial equation $P(A)=0$. One can see that the evaluation of a meromorphic function $F$ at $A$ is equal to $Q(A)$, where $Q$ is the degree $<n$ interpolation polynomial of $F$ with the the set of interpolation points equal to the set of roots of the polynomial $P$. 
\par In particular, for $A$ an $n \times n$ matrix, there is a common belief that for computing $F(A)$  one has to reduce $A$ to its Jordan form. Let $P$ be the characteristic polynomial of $A$. Then by the Cayley-Hamilton theorem, $P(A)=0$. And thus the matrix $F(A)$ can be found without reducing $A$ to its Jordan form.  Computation of the Jordan form for $A$ involves many extra computations. In the paper we show that it is not needed. One application is to compute the matrix exponential for a matrix with repeated eigenvalues, thereby solving arbitrary order linear differential equations with constant coefficients.
\end{abstract}


\setlength{\parindent}{0pt}
\setlength{\parskip}{1.6ex}

\pagestyle{headings}

\section{Introduction}

Interpolation polynomials with multiple interpolation points are widely used in applied mathematics under the name Hermite--Lagrange interpolation polynomials for approximating functions or other data sets by polynomials, see for example Chapter 3 in \cite{numeric}. But such polynomials are useful in pure mathematics as well. For example, one can construct Galois theory including the problem of solvability by radicals using Lagrange interpolation polynomials with simple roots as the main tool (see \cite{prof}).

One may guess that Lagrange interpolation polynomials with multiple interpolation points also have applications in pure mathematics. In the paper, we present our reconstruction of the theory of interpolation polynomials with multiple interpolation points and its applications to linear algebra.

This paper was written by Sushil Singla and Aaron Tronsgard who attended Askold Khovanskii's course on Topological Galois theory at the Fields Institute during the Fall, 2021. They brought to life a sketch of the theory presented on the course. 

Let $\mathbb K$ be a field of characteristic zero. Let $\mathbb K[x]$ be the polynomial algebra over $\mathbb K$. Let $\{\lambda_1,\dots, \lambda_k\}=\Lambda\subseteq \mathbb K$ be a set of $k$ distinct elements. For all $1\leq j\leq k$, let $m_j\in \bN$ be a natural number associated with $\lambda_j$ such that $\sum\limits_{j=1}^k m_j=n$.

\begin{definition} A polynomial $L\in\mathbb K[x]$ of degree less than $n$  is called the  \emph{Lagrange interpolation polynomial} with interpolation points $\lambda_1. \dots, \lambda_k$ with multiplities $m_1, \dots, m_k$ and the interpolation data $$c_1^{(0)},\dots,c_1^{(m_1-1)}, \dots, c_k^{(0)},\dots,c_k^{(m_k-1)}$$ if for every $\lambda_j\in\Lambda$ and $0\leq m< m_j$, we have $$L^{(m)}(\lambda_j)=c_j^{(m)},$$ where $L^{(0)}(x)=L(x)$ and for $m>0$, $L^{(m)}(x)$ denotes the $m^{th}$ derivative of $L$.
\end{definition}

We are justified in our language, defining the Lagrange interpolation polynomial due to the following.

\begin{theorem}\label{main} The Lagranage interpolation polynomial with given interpolation points and interpolation data exists and is unique.
\end{theorem}
\begin{proof} Let $L(x)=\sum\limits_{i=0}^{n-1} a_i x^i$ be polynomials with undetermined cofficients $a_0, \dots, a_{n-1}$. Now determining an interpolation polynomial is equivalent to solving the $n$ equations $$L^m(\lambda_j)=c_j^{(m)}.$$
This system of equations will have a unique solution provided the corresponding homogenous system $$L^m(\lambda_j)=0$$ has only the trivial solution.

A solution of the corresponding homogeoneous equation is a polynomial of degree less than $n$ such that for every $\lambda_j\in\Lambda$ and $0\leq m< m_j$, we have $L^{(m)}(\lambda_j)=0$ i.e. a polynomial $L$ which has roots $\lambda_1, \dots, \lambda_k$ with multiplities $m_1, \dots, m_k$. By definition, $\sum\limits_{j=1}^k m_j=n$. Clearly the zero polynomial satisfies these conditions. And this is only solution because if a polynomial of degree less than $n$ has $n$ roots counted with multiplicity, then it is identically zero.
\end{proof}

Note that if all of the multiplicities are equal to 1, we have an explicit formula for the Lagrange interpolation polynomial:
\begin{equation} \label{explicit} L(x) = \sum_{j=1}^n c_j \prod_{i \neq j} \frac{x - \lambda_i}{\lambda_j - \lambda_i} \end{equation}

It will be useful for us to specify the interpolation data by giving function values, rather than as a list of points. We may also want to specify the interpolation points via eigenvalues of a linear transformation. To this end, we make the following definitions.

\begin{definition} Let $Q: \mathbb{K}\rightarrow \mathbb{K}$ be a polynomial. The Lagrange interpolation polynomial of $Q$ with the interpolation points $\lambda_1, \dots, \lambda_k$ with multiplicities $ m_1, \dots, m_k$ is defined as the unique polynomial $L$ of degree less than or equal to $n := m_1+\dots+m_k$ such that for all $1\leq j\leq k$ and $0\leq m< m_j$, we have $$L^{(m)}(\lambda_j) = Q^{(m)}(\lambda_j).$$
\end{definition}

\begin{definition} Let $f: U\rightarrow \mathbb{C}$ be a function defined on an open subset $U$ of $\bC$. The Lagrange interpolation polynomial of $f$ with the interpolation points $\lambda_1, \dots, \lambda_k$ with multiplicities $m_1, \dots, m_k$ is defined as the unique polynomial $L$ of degree less than $n := m_1+\dots+m_k$ such that for all $1\leq j\leq k$ and $0\leq m< m_j$, we have $$L^{(m)}(\lambda_j) = f^{(m)}(\lambda_j),$$ where  $f^{(0)}(x)=f(x)$ and for $m>0$, $f^{(m)}(x)$ denotes the $m^{th}$ derivative of $f$.
\end{definition}

\textbf{Example 1.1}. If $f: U\rightarrow \mathbb{C}$ is an $(m-1)$-times differentiable function. Then for any $x_0\in U$, the Lagrange interpolation polynomial $L$ of $f$ with one interpolation point $x_0$ of multiplicity $m$ coincides with the degree $m-1$ Taylor polynomial of $f$ at the point $x_0$, that is, we have, $$L(x) = f(x_0)+f^{(1)}(x_0)(x-x_0)+\dots+\dfrac{1}{(m-1)!}f^{(m-1)}(x_0)(x-x_0)^{m-1}.$$

One very classical result on Taylor polynomials is the Lagrange remainder theorem. It states that if $f$ is an $n$-times differentiable function and $T_{n-1}$ denotes the degree $n-1$ Taylor polynomial of $f$ at point $x_0$, then for any $x\in\mathbb R$ there exists $\xi$ between $x$ and $x_0$ such that $$f(x) - T_{n-1}(x) =  \frac{f^{(n)} (\xi)}{n!} (x - x_0)^{n}.$$

It is a natural question whether we have an analogous result for the general interpolation polynomials. The answer is a resounding yes.

\begin{theorem}
	Let $f(x)$ be an $n$-times differentiable function. Let $L$ be its interpolation polynomial with interpolation points $\lambda_1, \ldots, \lambda_k$ and multiplicities $m_1, \ldots, m_k$ such that $\sum\limits_{j=1}^n m_j = n$. Then for $x_0 \in \mathbb{R}$,  there exists $\xi$ inside the convex hull $U$ formed by the interpolation points such that
	\[f(x_0) - L(x_0) = \frac{f^{(n)} (\xi)}{n!} (x_0 - \lambda_1)^{m_1} \ldots (x_0 - \lambda_k)^{m_k}. \]
\end{theorem}

\begin{proof}
	Consider the interpolation polynomial $L_0$ for $f$ consisting of the same data as $L$, with the additional interpolation point $\lambda_{k+1} = x_0$ with multiplicity one. We write
	$$ L_0 (x) = L(x) + C (x - \lambda_1)^{m_1} \ldots ( x - \lambda_k)^{m_k} $$
	with $$C = \frac{f(x_0) - L(x_0)}{(x_0 - \lambda_1)^{m_1} \ldots (x_0 - \lambda_k)^{m_k}}.$$
	
	Now, consider the function $f(x) - L_0(x)$. This function has at least $n+1$ roots in $U$, counting with multiplicity. By Rolle's theorem, this implies that the $n$-th derivative $(f - L_0)^{(n)}$ has at least one root in $U$, call it $\xi$. Moreover, we compute $$(f - L_0)^{(n)} (x) = f^{(n)}(x) - C n !.$$
	
	And therefore, we have $C = \frac{f^{(n)}(\xi)}{n!}$. And since $$f(x_0) - L(x_0) = C (x_0 - \lambda_1)^{m_1} \ldots (x_0 - \lambda_k)^{m_k},$$ this implies that $$f(x_0) - L(x_0) = \frac{f^{(n)}(\xi)}{n!}(x_0 - \lambda_1)^{m_1} \ldots (x_0 - \lambda_k)^{m_k}.$$
\end{proof}

In Section \ref{section3}, we show the application of Lagrange interpolation polynomial in computing functions of matrices. From Theorem \ref{4}, it follows that if $A$ is a  $n\times n$ matrix with the eigenvalues $\lambda_1, \ldots, \lambda_k$ with multiplicities $m_1, \ldots, m_k$, then for a rational function without poles at $\lambda_j$ for all $1\leq j\leq k$ or entire function $f$, we have $f(A)=Q(A)$ where $Q$ is the Lagrange interpolation polynomial of $f$ with interpolation points $\lambda_1, \dots, \lambda_k$ and multiplicities  $m_1, \dots, m_k$. And application of Lagrange interpolation polynomials in computing inverse of matrices and solution of homogeneous linear differential equations is shown. In Section \ref{section2}, principle Lagrange resolvents are defined and its applications are shown. Theorem \ref{neww} shows that for computation of matrices, it is enough to compute principle Lagrange resolvents. Finally, we end our discussions with application to linear algebra, specifically we provide a proof that all matrices can be put into Jordan normal form. In Section \ref{remarks}, a few remarks are mentioned.

\section{Computation of functions of matrices}\label{section3}

If interpolation points are roots of a polynomial, we have the folllowing proposition.

\begin{prop} Let $T\in\mathbb{K}[x]$ be a polynomial of degree $n$ with $k$ distinct roots $\lambda_1, \dots, \lambda_k$ with multiplitcities $m_1, \dots, m_k$ and $\sum\limits_{j=1}^k m_j=n$.

Then for any polynomial $Q\in\mathbb{K}[x]$ of degree at least $n$, a polynomial $L$ is Lagrange interpolation polynomial of $Q$ with interpolation points equal to roots of $T$ with corresponding multiplicities if and only if $$Q-L\equiv 0\ (mod\ T).$$
In other words, the Lagrange polynomial $L$ of $Q$ with the above interpolation data is the remainder of $Q$ by $T$.
\end{prop}

\begin{proof} Let $L\in\mathbb K[x]$ be remainder of $Q$ by $T$ i.e. there exists $S\in\mathbb K[x]$ such that $Q=TS+L$ and degree of $L$ is less than or equal to $n$.  For $1\leq j\leq k$ and $1\leq m< m_j$, we have $Q^{(m)}(\lambda_j) = \sum\limits_{i=1}^{m} T^{(i)}(\lambda_j)S^{(m-i)}(\lambda_j)+L^{(m)}(\lambda_j)$. Since $T^{(i)}(\lambda_i)=0$ for all $1\leq i\leq m$, we get $Q^{(m)}(\lambda_j) =L^{(m)}(\lambda_j)$. And by the definition and uniqueness of the Lagrange interpolation polynomial, we get the result.
\end{proof}

As an immediate application of the above proposition, we get the following theorem related to the computation of polynomials of operators.

\begin{theorem}\label{1} Let $A$ be a linear operator on a vector space $V$ over $\mathbb K$ such that $A$ satisfies a polynomial $T$ that splits over $\mathbb K$. Let $\lambda_1, \dots, \lambda_k$ be distinct roots of $T$ with multiplities $m_1, \dots, m_k$. Then for any $Q\in\mathbb{K}[x]$, we have $$Q(A)=R(A),$$ where $R$ is the Lagrange interpolation polynomial of $Q$ with interpolation points $\lambda_1, \dots, \lambda_k$ and multiplicities  $m_1, \dots, m_k$. (Note that $R$ is also the remainder of $Q$ by $T$.)
\end{theorem}

\begin{proof}
	By the previous proposition, we know that the Lagrange interpolation polynomial $R$ for $Q$ with interpolation points $\lambda_1, \ldots, \lambda_k$ and multiplicities $m_1, \ldots, m_k$ is the remainder of the division of $Q$ by $T$. In particular
	\[ Q - R \equiv 0 \mod (T). \]
	We have $T(A) = 0$, and therefore $(Q - R)(A) = 0$. Or, in other words
	\[ Q(A) = R(A). \]
\end{proof}

\begin{cor}
	Let $A$ be an $n\times n$ matrix with the eigenvalues $\lambda_1, \ldots, \lambda_k$ with multiplicities $m_1, \ldots, m_k$, such that $\sum\limits_{j=1}^n m_j = n$.
	
	Then for any polynomial $Q$ with degree at least $n$, we have
	\[ Q(A) = R(A) \]
	where $R$ is the interpolation polynomial of $Q$ with respect to the interpolation points $\lambda_1, \ldots, \lambda_k$ and multiplicities $m_1, \ldots, m_k$.
\end{cor}

\textbf{Example 2.1}. Let $A$ be $n\times n$ matrix with $n$ distinct eigenvalues $\lambda_1, \dots, \lambda_n$. Let $Q\in\mathbb{K}[x]$. Since the unique Lagrange interpolation polynomial $L$ of degree less than equal to $n$ with interpolation points $\lambda_1, \dots, \lambda_n$ and interpolation data $Q(\lambda_1), \dots, Q(\lambda_n)$ is $$L(x)=\sum\limits_{j=1}^n Q(\lambda_j) \prod\limits_{i\neq j} \dfrac{(x-\lambda_i)}{(\lambda_j-\lambda_i)} .$$ Hence we have, $Q(A)=L(A) = \sum\limits_{j=1}^n Q(\lambda_j) \prod\limits_{i\neq j} \dfrac{(A-\lambda_i)}{(\lambda_j-\lambda_i)} .$

We now restrict ourselves to the case $\mathbb{K} = \mathbb{C}$. And we get following results for the functions of operators.

\begin{theorem}\label{4} Let $A$ be an operator on a vector space $V$ over $\mathbb C$ such that $A$ satisfies a polynomial $T\in\mathbb C[x]$. Let $\lambda_1, \dots, \lambda_k$ be distinct roots of $T$ with multiplicities $m_1, \dots, m_k$. The following are true.
\begin{enumerate}
\item Assume $0\neq\lambda_j$ for all $1\leq j\leq k$. Then $A$ is invertible and we have $$A^{-1}=L(A),$$ where $L$ is the Lagrange interpolation polynomial of  the function $$1/x$$ with interpolation points $\lambda_1, \dots, \lambda_k$ and multiplicities $m_1, \dots, m_k$.
\item Consider a rational function $f(x) = P(x)/Q(x)$ such that $Q$ does not vanish at $\lambda_j$ for all $1\leq j\leq k$. Let $L$ be an interpolation polynomial of $f$ with the interpolation  points $\lambda_1, \dots, \lambda_k$ and multiplicities $m_1, \dots, m_k$. Then the operator $P(A)[Q(A)]^{-1}$ is defined and it is equal to $L(A)$.
\item For an entire function $F(x)$ of complex variable $x$, we have $$F(A)=L(A),$$ where $L$ is the Lagrange interpolation polynomial of  the function $F$ with interpolation points $\lambda_1, \dots, \lambda_k$ and multiplicities $m_1, \dots, m_k$.
\end{enumerate}
\end{theorem}
\begin{proof}
	\begin{enumerate}[label=(\roman*)]
	\item We know that $\lambda_1, \ldots, \lambda_k$ are roots of the rational function
	\[ \frac{1}{x} - L(x) \]
	with multiplicities $m_1, \ldots, m_k$. Multiplying by $x$, we retain those roots and so we have
	\[ 1 - xL(x) \equiv 0 \mod (T). \]
	And since $T(A) = 0$ this implies that
	\[ A L(A) = I. \]
	\item Since $Q$ does not vanish on $\Lambda$ we know that $(Q,T) = 1$, and thus we can find $V, U\in \mathbb C(x)$ such that
	\[Q(x) V(x) + U(x) T(x) = 1. \]
	Evaluating at $A$ we get
	\[Q(A) V(A) = I. \]
	So $Q(A)$ is invertible, with $Q^{-1}(A) := V(A)$. Then, proceeding similarly to the case above, we have that
	\[ P(x)  - L(x) Q(x) \equiv 0 \mod (T). \]
	And therefore
	\[ P(A) = L(A) Q(A). \]
	Finally, we multiply on the right by $Q(A)^{-1}$ to get
	\[ L(A) = P(A) Q(A)^{-1}. \]
	
	\item We know that for $1 \leq j \leq k$ and $0 \leq m \leq m_j-1$ we have
	\[F^{(m)} (\lambda_j) = L^{(m)} (\lambda_j) \]
	Therefore, we can write
	\[ F(x) - L(x) = T(x) G(x) \]
	for some entire function $G(x)$. Evaluating at $A$ we get
	\[F(A) - L(A) = 0. \]

\end{enumerate}
\end{proof}

\textbf{Example 2.2}. Let $A$ be $n\times n$ matrix with $n$ distinct eigenvalues $\lambda_1, \dots, \lambda_n$. From Example 2.1 and Theorem \ref{4}, we have
\begin{equation} \label{2} \text{exp}(A) = \sum\limits_{j=1}^n \text{exp} (\lambda_j) \prod\limits_{i\neq j} \dfrac{(A-\lambda_i)}{(\lambda_j-\lambda_i)}. \end{equation}
In fact, for any entire complex valued function $f$, we have $$f(A) = \sum\limits_{j=1}^n f(\lambda_j) \prod\limits_{i\neq j} \dfrac{(A-\lambda_i)}{(\lambda_j-\lambda_i)} .$$

Below we show application of our methods to compute inverse of a $3\times 3$ matrix and  in solving an order 3 homogeneous linear differential equation.

\subsection*{Computing the Inverse of a matrix}

As seen in Theorem \ref{4}, we are able to compute the inverse of a non-singular matrix $A$ using a Lagrange interpolation polynomial. Thus, we are able to solve consistent systems of linear equations. Let
$$ A = \begin{pmatrix}
	9 & -15 & -25 \\
	1 & 0 & 0 \\
	0 & 1 & 0
\end{pmatrix}$$
We compute the characteristic polynomial for $A$: $T(x) = - (x+1)(x-5)^2$. Thus, we want to find the Lagrange interpolation polynomial $L(x)$ for $f(x) = 1/x$, interpolation points $\lambda_1 = 5$, $\lambda_2 = -1$ with multiplicities $m_1 = 2$ and $m_2 = 1$.

There is a useful trick for computing the interpolation polynomial with \textit{one more} multiplicity than one we already know (which we do for the case $m_1 = m_2 = 1$). Let $$L_0 (x) = \frac{x-5}{6} + \frac{x+1}{30}$$
denote the interpolation polynomial with simple multiplicities. Then, we look for $L(x) = L_0(x) + c (x-5)(x+1)$ where $c \in \mathbb{C}$ is a constant.

In our case, we evaluate the derivative at 5 and require it equal to $f'(5) = \frac{-1}{25}$: $$L'(5) = \frac{1}{6} + \frac{1}{30} + 6c$$
from which we find $c = \frac{-1}{25}$. And so $L(x) = \frac{x-5}{6} + \frac{x+1}{30} - \frac{1}{25}(x-5)(x+1)$. The relevant matrices are
\begin{center}
$\begin{array}{cc}
	A + I =
		\begin{pmatrix}
		10 & - 15 & -25 \\
		1 & 1 & 0 \\
		0 & 1 & 1
		\end{pmatrix} &
		A - 5I =
		\begin{pmatrix}
			4 & -15 & -25 \\
			1 & -5 & 0 \\
			0 & 1 & -5
		\end{pmatrix}
\end{array}$
\end{center}

$$(A+I)(A-5I) =
\begin{pmatrix}
	25 & -100 & -125 \\
	5 & -20 & -25 \\
	1 & -4 & -5
\end{pmatrix}$$

And we find that $$A^{-1} = L(A) = \begin{pmatrix}
	0 & 1 & 0 \\
	0 & 0 & 1 \\
	-\frac{1}{25} & \frac{9}{25} & -\frac{3}{5}
\end{pmatrix}$$

\subsection*{An order 3 homogeneous linear differential equation}

One standard application of the matrix exponential is to ordinary linear differential equations. In particular, a general solution to the matrix differential equation
$$ \mathbf{y}' = A \mathbf{y} $$
is given by $\mathbf{y}(t) = \exp(t A) \mathbf{c}$, where $\mathbf{c}$ is an arbitrary column of constants.

The standard solution to this problem is to fix a basis of the vector space so that the matrix $A$ is in Jordan normal form, and to learn how to take the exponentials of Jordan blocks. See for example Chapters 5 and 6 of \cite{DE}. Similar to the above computation of the inverse matrix, Theorem \ref{4} allows us to compute the matrix exponential using Lagrange interpolation polynomials, without considering any special bases, and without needing any normal forms.

Consider for example the equation $$y''' - 9 y'' + 15 y' + 25y = 0.$$ Written as a system of differential equations, this is
\[ \mathbf{y}' = \begin{pmatrix}
	9 & -15 & -25 \\
	1 & 0 & 0 \\
	0 & 1 & 0
\end{pmatrix} \mathbf{y} \]
where $\mathbf{y} = \begin{pmatrix}
	y'' \\
	y' \\
	y
\end{pmatrix}$, and we see the same matrix $A$ as in the previous example. We are interested in computing $\exp (tA)$. Similar to before, the characteristic polynomial is $T(x) = - (x+t)(x-5t)^2$, so we want to compute the interpolation polynomial for the function $f(x) = \exp(x)$ with interpolation points $\lambda_1 = 5t$ and $\lambda_2 = -t$ with multiplicities 2 and 1 respectively. Again, we write down the interpolation polynomial with simple multiplicities $$L_0 (x) = -\frac{e^{-t} (x - 5t)}{6t} + \frac{e^{5t} (x+t)}{6t}$$ and we look for $L(x) = L_0(x) + c (x-5t)(x+t)$, noting that $c$ may depend on $t$ (but not $x$).

We require $$e^{5t} = f'(5t) = L'(5t) = -\frac{e^{-t}}{6t} + \frac{e^{5t}}{6t} + 6ct, $$ from which we find $ct^2 = \frac{e^{5t}(6t-1) + e^{-t}}{36}$. And therefore
\[ L(x) = -\frac{e^{-t} (x - 5t)}{6t} + \frac{e^{5t} (x+t)}{6t} + \frac{e^{5t}(6t-1) + e^{-t}}{36t^2} (x+t)(x-5t). \]

The relevant matrices are
$$\begin{array}{cc}
	tA - 5tI =
	\begin{pmatrix}
		4t & - 15t & -25t \\
		t & -5t & 0 \\
		0 & t & -5t
	\end{pmatrix} &
	tA + tI =
	\begin{pmatrix}
		10t & -15 t& -25t \\
		t & t & 0 \\
		0 & t & t
	\end{pmatrix}
\end{array}$$
$$(tA+tI)(tA-5tI) =
\begin{pmatrix}
	25t^2 & -100t^2 & -125t^2 \\
	5t^2 & -20t^2 & -25t^2 \\
	1t^2 & -4t^2 & -5t^2
\end{pmatrix}$$

Evaluating $L(tA) = \exp(tA)$, we find that the general solution is given by
\begin{equation*}
	\begin{split}
		y(t) = C_1 &\left( -\frac{2e^{-t}}{3} + 2 e^{5t} + \frac{25}{36} \left( e^{5t} (6t-1) + e^{-t} \right) \right) \\
		&+ C_2 \left( -\frac{5e^{-t}}{2} - 3 e^{5t} - \frac{25}{9} \left( e^{5t} (6t-1) + e^{-t} \right) \right) \\
		&+ C_3 \left( -\frac{25e^{-t}}{6} - 5 e^{5t} - \frac{125}{36} \left( e^{5t} (6t-1) + e^{-t} \right) \right)
	\end{split}
\end{equation*}
where $C_1, C_2, C_3 \in \mathbb{C}$ are arbitrary constants.

The above examples show that the general procedure to compute the Lagrange interpolation polynomials is as follows. If we know the interpolation polynomial $L_0$ for $k$ points $\lambda_1, \ldots, \lambda_k$  with multiplicities $m_1, \ldots, m_k$. Then to find interpolation polynomial $L$ for the same data of interpolation polynomial with one extra multiplicity or one extra point with multiplicity $1$ can be computed by taking $$L(x)=L_0(x)+c(x - \lambda_1)^{m_1} \cdots (x - \lambda_k)^{m_k}$$ and using the interopolation data to find $c$.

The following result shows the method of computing the interpolation polynomial on the union of two sets of interpolation points:

\begin{theorem}
	Let $T_1(x) = (x - \lambda_1)^{m_1} \ldots (x - \lambda_p )^{m_p}$ and $T_2 (x) = (x - \beta_1)^{\ell_1} \ldots (x - \beta_q)^{\ell_q}$ with $\lambda_i\neq \beta_j$ for all $1\leq i\leq p$ and $1\leq \beta_j\leq q$.
	Then the Lagrange interpolation polynomial of $Q \in\mathbb K[x]$ with respect to $\Lambda = \Lambda_1 \cup \Lambda_2$ is equal to
	\[ L = [Q T_2^{-1}]_1 T_2 + [Q T_1^{-1}]_2 T_1 \]
	where $T_1^{-1}, T_2^{-1}$ denotes the inverse of $T_1$ and $T_2$ in the ideal generated by $T_2$ and $T_1$ respectively (which exists because $T_1$ and $T_2$ has no common roots). And $[QT_2^{-1}]_1$ and $[QT_1^{-1}]_2$ denote the interpolation polynomials of $QT_2^{-1}$ and $QT_1^{-1}$, as rational functions, with respect to $\Lambda_1$ and $\Lambda_2$ respectively.
\end{theorem}

\begin{proof} Let $n_1=\sum\limits_{j=1}^p m_j$ and $n_2=\sum\limits_{j=1}^q l_j$.  First observe that as defined, $L$ is a polynomial of degree smaller than $n_1 + n_2$. This is easily seen since $[\cdot]_1$ and $[\cdot]_2$ have degrees smaller than $n_1$ and $n_2$ respectively.
	
	Let us examine what happens at $\Lambda_1$. We know that
	\[ [QT_2^{-1}]_1 - \frac{Q}{T_2} \]
	has roots of multiplicities $m_1, \ldots, m_q$ at $\lambda_1, \ldots, \lambda_q \in \Lambda_1$. In particular, write
	\[ [QT_2^{-1}]_1 (x) = \frac{Q(x)}{T_2(x)} + (x - \lambda_1)^{m_1} \ldots (x - \lambda_q)^{m_q} G(x) \]
	for some rational function $G$ which has no poles on $\Lambda_1$. Then we have
	\[ L(x) = Q(x) + (x - \lambda_1)^{m_1} \ldots (x - \lambda_q)^{m_q} G(x) T_2(x) + [QT_1^{-1}]_2 (x) T_1 (x) \]
	Since $T_1$ has roots on $\Lambda$ of the corresponding multiplicities, we see right away that for $1 \leq j \leq q$:
	\[ L^{(m)} (\lambda_j) = Q^{(m)} (\lambda_j). \]
	for all $0 \leq m \leq m_j - 1$. Symmetric arguments will give that for $1 \leq j \leq q$:
	\[L^{(\ell)} (\beta_j) = Q^{(\ell)} (\beta_j) \]
	for all $0 \leq \ell \leq \ell_j - 1$.
	
	And so by uniqueness of the Lagrange interpolation polynomial we are done.
\end{proof}

\section{Principle Lagrange resolvent and its applications}\label{section2}

In this section, we define principle Lagrange resolvents and we show that it is enough to compute the principle Lagrange interpolation polynomials to compute functions of matrices over $\mathbb C$.

As mentioned before that the standard method to solve the homogenous linear ordinary differential equations is in finding the Jordan normal form of matrices and  learn how to take the exponentials of Jordan blocks. We have already seen the application of Lagrange interpolation in solving homogeneous linear ordinary differential equations without computing Jordan canonical form. As an application of principle Lagrange interpolation polynomials, we  provide a proof that all matrices can be put into Jordan normal form.

\begin{definition} In the special case where $c_i^{(0)}=1$ and all other interpolation data are $0$, we define the corresponding Lagrange interpolation polynomial to be the  \emph{principal Lagrange resolvent} with respect to $\lambda_i$.
\end{definition}
We will denote the principal Lagrange resolvent with respect to $\lambda_i$, by $\hat{T}_i$.

\begin{prop} \label{properties}
	For $\{\hat{T}_i \}_{i=1}^k$ principal Lagrange resolvents on $k$ interpolation points with multiplicities $m_1, \ldots , m_k$ summing to $n$, we have
	\begin{enumerate}[label=(\roman*)]
		\item $\hat{T}_1 +  \ldots + \hat{T}_k - 1 = 0 $,
		\item $\hat{T}_i \hat{T}_j \equiv 0 \mod (T) \text{ if } i \neq j$,
		\item $(\hat{T}_i )^2 \equiv \hat{T}_i \mod (T)$,
		\item $(t - \lambda_i)^{m_i} \hat{T}_i \equiv 0 \mod (T)$ for all $1\leq i\leq k$.
	\end{enumerate}
\end{prop}

\begin{proof}
	\begin{enumerate}[label=(\roman*)]
		\item Note that $\lambda_1, \ldots, \lambda_k$ are distinct roots with multiplicities $m_1, \ldots m_k$ of the degree at most $(n-1)$ polynomial
		\[ \hat{T}_1 + \ldots + \hat{T}_k - 1. \]
		Since $\sum\limits_{i=1}^k m_i = n$, this implies that the polynomial is identically 0.
		
		\item Let $1 \leq i < j \leq k$. Then, for every $\lambda_\ell$, $1 \leq \ell \leq k$ we have either $\hat{T}_i(\lambda_\ell) = 0$ or $\hat{T}_j (\lambda_\ell) = 0$ with multiplicity $m_\ell$. And so every root of $T$ is a root of $\hat{T}_i \hat{T}_j$ with at least equal multiplicity. And therefore $\hat{T}_i \hat{T}_j \equiv 0 \mod (T)$.
		
		\item Consider the polynomial $(\hat{T}_i)^2 - \hat{T}_i = \hat{T}_i ( \hat{T}_i - 1)$. For $\ell\neq i$, $\lambda_{\ell}$ is a root of $\widehat{T_i}$ with multiplicity atleast $m_{\ell}$, and therefore a root of multiplicity at least $m_{\ell}$ of $\hat{T}_i^2 - \hat{T}_i$. For $\lambda_{i}$, we have $\widehat{T_i}(\lambda_i)=1$ and $\widehat{T_i}^m(\lambda_i)=0$ for all $m\in\bN$. Now, $$(\widehat{T_i}^2-\widehat{T_i})^{m}(\lambda_i) = \sum\limits_{p=0}^m \widehat{T_i}^p(\lambda_i)(\widehat{T_i}-1)^{m-p}(\lambda_i) = 0.$$ So, $\lambda_i$ is also root of  $\widehat{T_i}^2-\widehat{T_i}$ with multiplicity at least $m_i$. Therefore, we have $\hat{T}_i^2 - \hat{T}_i \equiv 0 \mod (T)$.
		
		\item Similarly, note that $\lambda_j$, $j \neq i$ is a root of multiplicity at least $m_j$ for $\hat{T}_i$. And clearly, $\lambda_i$ is a root of $(t - \lambda_i)^{m_i}$ with multiplicity $m_i$. Therefore $(t - \lambda_i)^{m_i} \hat{T}_i \equiv 0 \mod (T)$.
	\end{enumerate}
\end{proof}

Consider now a linear operator $A$ over a vector space $V$ (possibly infinite dimensional) over $\mathbb K$. Suppose $A$ satisfies a polynomial $T\in \mathbb K[x]$ of degree $n$. Assume $T$ splits over $\mathbb K$, and has $k$ different roots $\lambda_1,\dots, \lambda_k$ with multiplicities $m_1, \dots, m_k$.

\begin{definition} The operator $L_i(A)=\widehat{T_i}(A)$ is called principle Lagrange resolvent of the operator $A$ corresponding to the polynomial $T$ and root $\lambda_i$, where $\widehat{T_i}(x)$ is the principle Lagrange resolvent corresponding to $\lambda_i$.
\end{definition}

\begin{theorem}\label{main2} The principle Lagrange resolvents $L_i$ of the operator $A$ corresponding to an annihilating polynomial $T$ satisify the following.
	\begin{enumerate}[label=(\roman*)]
		\item\label{one} $L_1(A)+\dots+L_k(A)= I$, where $I$ is identity matrix,
		\item $L_i(A)L_j(A)=0$ for $i\neq j$,\label{two}
		\item $L_i^2(A)=L_i(A)$,\label{three}
		\item $(A-\lambda_i I)^{m_i} L_i(A)=0$.\label{four}
	\end{enumerate}
\end{theorem}

\begin{proof}
	Each of these expressions is obtained directly from evaluating the corresponding polynomials on the left hand sides of Proposition \ref{properties} at $A$. Each of these polynomials is divisible by $T$, and $T$ annihilates $A$, so we are able to replace the $\mod (T)$ congruence with equality.
\end{proof}

As a direct corollary of Theorem \ref{4} and Example 2.2, we get that for a matrix $A$ over $\mathbb C$ having only one eigenvalue, functions of matrices can be computed by principle Lagrange resolvents.

\begin{cor}\label{new} Let $A$ be an operator on a vector space $V$ over $\mathbb C$ such that $A$ satisfies the polynomial $(x-\lambda_1)^n$. Let $F(x)$ be an entire function or a rational function such that $\lambda_1$ is not a pole of $F(x)$, then $$F(A)=T_{\lambda_1}^{(n-1)}(A),$$ where $T_{\lambda_1}^{(n-1)}(x)$ is the degree $n-1$ Taylor polynomial of $F$ at point $\lambda_1$. Thus, $$F(A) = T_{\lambda_1}^{(n-1)}(A) = \sum\limits_{k=0}^{n-1}\dfrac{1}{k!} F^{(k)}(\lambda_1)(A-\lambda_1 I)^{k}.$$
\end{cor}

Using Theorem \ref{main2}, we get the following more general result.

\begin{cor}\label{neww} Let $A$ be an operator on a vector space $V$ over $\mathbb C$ such that $A$ satisfies a polynomial $T(x)\in\mathbb C[x]$. Let $\lambda_1, \dots, \lambda_k$ be distinct roots of $T$ with multiplities $m_1, \dots, m_k$. Let $F(x)$ be an entire function or a rational function such that $\lambda_1, \dots, \lambda_k$ are not poles of $F(x)$, then $$F(A)=\sum\limits_{i=1}^kT_{\lambda_i}^{(m_i-1)}(A)L_i(A),$$ where $T_{\lambda_i}^{(m_i-1)}(x)$ is the degree $m_i-1$ Taylor polynomial of $F$ at point $\lambda_i$ and $L_i(A)$ is the principle Lagrange resolvent of the operator $A$ corresponding to the polynomial $T$ and root $\lambda_j$.
\end{cor}

\begin{proof} Using Theorem \ref{main2}, we have $F(A)=F(A)L_1(A)+\dots+F(A)L_k(A)$. For every $1\leq i\leq k$, $F(A)L_i(A) = F(A_i) L_i(A)$, where $A_i$ is the restriction of  operator $A$ on the range $V_i$ of $L_i(A)$. Since $(A-\lambda_iI)^{m_i}(v)=0$ for all $v\in V_i$, the characteristic polynomial of $A_i$ is $(x-\lambda_i)^{m_i}$. Using Corollary \ref{new}, we get $F(A_i)=T_{\lambda_i}^{(m_i-1)}(A)$.
\end{proof}

\textit{Alternative proof:} The proof follows from Theorem \ref{4} and noting that the polynomial given by \begin{equation}\label{11}P(x)=\sum\limits_{i=1}^k\left(\sum\limits_{p=0}^{m_i-1} \dfrac{c_i^p}{p!} (x-\lambda_i)^p\right) L_i(x).\end{equation} has same interpolation data as that of the Lagrange interpolation polynomial $L$ of  the function $F$ with interpolation points $\lambda_1, \dots, \lambda_k$, multiplicities $m_1, \dots, m_k$ and interpolation data $$c_1^{(0)},\dots,c_1^{(m_1-1)}, \dots, c_k^{(0)},\dots,c_k^{(m_k-1)}$$ is

Now, we present a formula for the principle Lagrange interpolating polynomial with the given interpolation data.

\begin{prop} Let $\Lambda=\{\lambda_1,\dots, \lambda_k\}\subseteq \mathbb K$ be a set of $k$ distinct elements. For all $1\leq i\leq k$, let $m_i\in \bN$ be a natural number associated with $\lambda_i$ such that $\sum\limits_{i=1}^k m_i=n$. Then the principle Lagrange resolvent with respect to $\lambda_i$ is given by $$\widehat{T_i}=T_{\lambda_i}^{m_i-1}(x) \prod\limits_{j\neq i} (x-\lambda_j)^{m_j},$$ where $T_{\lambda_i}^{m_i-1}$ is the Taylor polynomial at $\lambda_i$ of degree $m_i-1$ of the function $$F(x)=\dfrac{1}{ \prod\limits_{j\neq i} (x-\lambda_j)^{m_j}}.$$
\end{prop}

\begin{proof} By definition we know that the principal Lagrange resolvent is the Lagrange interpolation polynomial with interpolation data $c_1^{(0)} = 1$ and everything else 0. We proceed by verifying that $\hat{T}_i$ as defined is this interpolation polynomial.
	
	First, note that $\deg \hat{T}_i = n-1$ as required. Then, for $j \neq i$ and $0 \leq m \leq m_j - 1$ we easily see that
	\[ \hat{T}_i^{(m)} (\lambda_j) = 0. \]
	And now we want to see what happens at $\lambda_i$. By definition of the Taylor polynomials we know that
	\[T_{\lambda_i}^{m_1 - 1} (x) - F(x) \]
	has a root at $\lambda_i$ of multiplicity $m_i$. In particular, we write
	\[ T_{\lambda_i}^{m_1 - 1} (x) - F (x) = (x - \lambda_i)^{m_i} G(x) \]
	for some $G(x)$ with no pole at $\lambda_i$. And then we have
	\[ \hat{T}_i (x) = F(x) \prod_{j \neq i} (x - \lambda_j)^{m_j} + G(x) \prod_{j=1}^k (x - \lambda_j)^{m_j} = 1 +  G(x) \prod_{j=1}^k (x - \lambda_j)^{m_j}. \]
	
	Since $G$ does not have a pole at $\lambda_i$, we see right away that $\hat{T}_i (\lambda_i) = 1$ and
	\[ \hat{T}_i^{(m)} (\lambda_i) = 0 \]
	for all $1 \leq m \leq m_i - 1$.
	
	And therefore, $\hat{T}_i$ is indeed the Lagrange interpolation polynomial with data $c_i^{(0)} = 1$ and all others 0.

\end{proof}

\begin{definition} For every vector $v\in V$, the vector $v_i = L_i(A)v$ will be called the principal Lagrange resolvent of $v$ corresponding to the root $\lambda_i$ and the operator  $A$.
\end{definition}

Then we get the following corollary to Theorem \ref{main2}:

\begin{cor}\label{main3} Every vector $v\in V$ is representable as sum of its principle Lagrange resolvents i.e. $v=v_1+\dots+v_k$. Moreover, all non zero Lagrange resolvents of $v$ are linearly independent and satisies $(A-\lambda_i I)^{m_i}v_i=0$.
\end{cor}

Now we restrict to a special case. Let $V$ be a finite dimensional space and $T$ be the characteristic polynomial of $A$ that splits over $\mathbb K$. By Cayley Hamilton theorem, $T(A)=0$. We consider principle Lagrange resolvents $L_i(A)$ of $A$ with resepct to $T$. Let $E_i$ denote the range of $L_i(A)$, that is, $E_i=\{L_i(A)v: v\in V\}$. Then we have the following.

\begin{theorem}\label{main4} The following are true.
	\begin{enumerate}[label=(\roman*)]
		\item\label{11} $V=E_1\oplus\dots\oplus E_k$,
		\item\label{12} $E_i =\{v\in V: (A-\lambda_i I)^{m_i}v=0\}=\{v\in V: \text{there exists } p\in\bN\text{ such that }(A-\lambda_i I)^{p}v=0\}$,
		\item\label{13} The dimension of $E_i$ is $m_i$, where $m_i$ is the multiplicity of the eigenvalue of $A$.
	\end{enumerate}
\end{theorem}
\begin{proof} Corollary \ref{main3} is equivalent to $V=E_1\oplus\dots\oplus E_k$ and $E_i\subseteq\{v\in V: (A-\lambda_i I)^{m_i}v=0\}$. Clearly, $V_i=\{v\in V: (A-\lambda_i I)^{m_i}v=0\}$ is a subspace of $V$ and also invariant under the action of $A$. So the characteristic polynomial of $A$ restricted to $V_i$ divides characteristic polynomial $T$ of $A$. For every $v\in V_i$, an annihilating polynomial of $v$ is of form $(t-\lambda_i)^p$, so the characteristic polynomial of $A$ restricted to $V_i$ is also of form $(t-\lambda_i)^p$. Since $(t-\lambda_i)^p$ divides $T$, we have $p\leq m_i$. So dimension of $E_i$ and $V_i$ are less than equal to $m_i$. Since $V=E_1\oplus\dots\oplus E_k$ and $\sum\limits_{i=1}^k m_i=n$, we must have dimension of $E_i$ and $V_i$ both equal to $m_i$. This proves $\labelcref{11}$ and $\labelcref{13}$. It also proves $E_i=V_i$. Now let $v\in V$ such that there exists $p\in\bN$ such that $(A-\lambda_i I)^{p}v=0$. Let $(t-\lambda_i)^p$ itself be the annihilating polynomial of $v$. By definition, $(t-\lambda_i)^p$ divides characteristic polynomial $T$ of $A$. Hence $p\leq m_i$. So $(A-\lambda_i I)^{m_i}v=0$. Hence $v\in E_i$. This completes the proof of the theorem.
\end{proof}

Now we are ready to prove Jordan decomposition theorem for matrices whose chracteristic polynomial splits. We need the following definition.

\begin{definition} Let $A$ be a linear operator on a vector space $V$. Let $v\in V$ such that there exists $m\in\bN$ such that $(A-\lambda I)^m(v)=0$, then $v$ is said to be a generalized eigenvector of $A$ corresponding to the eigenvalue $\lambda$. Suppose that $p$ is the smallest positive integer for which $(A - \lambda I)^p(v) = 0$. Then the set $\{v, (A - \lambda I)(v), (A - \lambda I)^2(v), \dots (A - \lambda I)^{p-1}(v)\}$ is called a cycle of generalized eigenvectors of $A$ corresponding to $\lambda$. generated by $v$.
\end{definition}
The proof of the following lemma is straightforward.
\begin{lemma} Let $v_1, \dots,v_k$ be linearly independent generalized eigenvector of $A$ corresponding to the eigenvalue $\lambda$. Then the union of sets of cycle generalized eigenvectors of generated by $v_i$ is linearly independent.
\end{lemma}
\begin{theorem} Let $A$ be a linear operator on a finite-dimensional vector space $V$ such that characteristic polynomial of $A$ splits. Then for every eigenvalue $\lambda$ of $A$, there exists a basis of $E_{\lambda}$ consisting of cycle generalized eigenvectors $A$ corresponding to $\lambda$.
\end{theorem}
\begin{proof} Consider the maximal set $\mathcal B$ of linearly independent generalized eigenvectors of $A$ corresponding to eigenvalue $\lambda$. And the claim is that it is basis for $E_{\lambda}$, else there will be a generalized eigenvectors not belonging to $E_{\lambda}\setminus\text{span}(\mathcal B)$, which will contradict maximality of $\mathcal B$.
\end{proof}

And finally using Theorem \ref{main4} and the basis of  $E_{\lambda}$ consisting of cycle generalized eigenvectors $A$ corresponding to $\lambda$, the matrix represntation of $A$ is the Jordan decomposition of matrix $A$.

\section{Remarks}\label{remarks}

\textbf{Remark 1}. Corollary \ref{neww} also holds for any meromorphic function $F$ such that $\lambda_1, \dots, \lambda_k$ are not poles of $F(x)$.

\textbf{Remark 2}. Let $A$ and $B$ be linear operators on a vector space $V$ over $\mathbb K$ such that $A$ and $B$ satisfies a polynomial $T(x)$ and $Q(x)$ respectively, that splits over $\mathbb K$. If we consider $L_{i, j} = \widehat{T_i}(A)\widehat{Q_j}(B)$, where $\widehat{T_i}(A)$ and $\widehat{Q_j}(B)$ are principle Lagrange resolvent of the operator $A$ and $B$ corresponding to the polynomial $T$ and root $\lambda_i$, and $Q$ with root $\mu_j$ respectively. Then we get that  Proposition 1.1.6 of \cite{prof} holds even when $T(x)$ and $Q(x)$ has roots with multiplicity greater than $1$. As an application, we get a direct proof of following theorem.

Let $A_1, \dots, A_n$ be commuting family of  linear operators on a vector space $V$ over $\mathbb K$ that satisfies polynomials that split over $\mathbb K$. Then $V$ can be written as direct sum of subspaces, which are simultaneously contained in generalized eigenspaces of $A_1, \dots, A_n$.

We note that the above theorem is a generalization of theorem that for a commuting family $\mathcal F$ of linear diagonalizable on a finite dimensional vector space $V$, there exists an ordered basis of $V$ in which all operators in $\mathcal F$ are represented by diagonal matrices.

\end{document}